\documentclass{article}
\usepackage{hyperref}
\usepackage{graphicx}
\usepackage[utf8x]{inputenc}
\usepackage[T1,T2A]{fontenc}
\usepackage[russian,english]{babel}
\usepackage{babelbib}

\usepackage{amsmath,amssymb,amsthm,amsfonts}

\newtheorem{theorem}{Theorem}

\newtheorem{example}{Example}

\newcommand{\gf}{\mathfrak{g}}

\begin{document}
\title{Limit shape of probability measure on tensor product of $B_n$ algebra modules}

\author{A.A.~Nazarov  $^1$, O.V.~Postnova  $^2$\\
  {\small $^1$ Department of High-Energy and Elementary Particle Physics,}\\
  {\small St.~Petersburg State University}\\
  {\small 198904, Ulyanovskaya 1, St. Petersburg, Russia}\\
  {\small e-mail: antonnaz@gmail.com}\\
  {\small$^{2}$ Laboratory of Mathematical Problems of Physics,}\\
  {\small St. Petersburg Department
    of Steklov Mathematical Institute}\\
  {\small of Russian Academy of Sciences}\\
  {\small 191023, Fontanka 27, St. Petersburg, Russia}\\
  {\small email: postnova.olga@gmail.com}
}

%\address{Department of High-energy and elementary particle physics,  St.Petersburg State University, 198904, Saint-Petersburg, Russia\\
%e-mail: antonnaz@gmail.com}
% * <antonnaz@gmail.com> 2016-10-14T13:05:35.351Z:

\maketitle
\begin{abstract}
 We study a probability measure on integral dominant weights in the decomposition of $N$-th tensor power of spinor representation of the Lie algebra $so(2n+1)$. The probability of the dominant weight $\lambda$ is defined as the ratio of the dimension of the irreducible component of $\lambda$ divided by the total dimension $2^{nN}$ of the tensor power. We prove that as $N\to \infty$ the measure weakly converges to the radial part of the $SO(2n+1)$-invariant measure on $so(2n+1)$ induced by the Killing form. Thus, we generalize Kerov's theorem for $su(n)$ to $so(2n+1)$.
  % \noindent{\it Keywords\/}: Lie algebra, tensor product, multiplicity, cluster algebra, grading
\end{abstract}

\section*{Introduction}
\label{sec:introduction} 
Let $V_1, V_2,\dots, V_{k}$ be finite dimensional representations of a simple Lie algebra $\gf$. The tensor product of these representations
is isomorphic to a direct sum of irreducible highest-weight
representations:

\begin{equation} 
V_1\otimes V_2\otimes\dots\otimes V_{k} \cong \bigoplus_{\lambda}W^{1,\dots,k}_{\lambda} \otimes L^{\lambda},
\end{equation}
where $L_{\lambda}$ is the irreducible representation with the highest weight $\lambda$ and
$W^{1,\dots,k}_{\lambda}\simeq \mathrm{Hom}_{\gf}(L^{\lambda},V_{1}\otimes\dots\otimes V_{k})$. The
dimension $M^{1,\dots,k}_{\lambda}=\mathrm{dim}W^{1,\dots,k}_{\lambda}$ is the multiplicity of
$L^{\lambda}$ in the tensor product decomposition. Rewriting this decomposition
equation in terms of dimensions of representations and dividing
by $\mathrm{dim}\cdot V_1\cdot\mathrm{dim}V_2\cdot\dots\cdot\mathrm{dim}V_{k}$, we get:
\begin{equation}
\sum_{\lambda} \frac{M^{1,\dots,k}_{\lambda} \mathrm{dim}L^{\lambda}}{\mathrm{dim}V_1\dots
\mathrm{dim}V_k}=1
\end{equation}
Thus, we have a discrete probability measure on the dominant
integral weights $\lambda$ that appear in this decomposition:
\begin{equation}
  \label{eq:1}
   \mu_N(\lambda) =  \frac{M^{1,\dots,k}_{\lambda} \mathrm{dim}L^{\lambda}}{\mathrm{dim}V_1\dots
\mathrm{dim}V_k}.
\end{equation}

Particularly interesting case of such measure appears when we take the $N$-th tensor
power of vector representation of $su(n+1)$. Because of the Schur-Weyl
duality, in this case $M_{\lambda}$ is the dimension of the irreducible representation of the symmetric group $S_N$. Kerov \cite{kerov1986asymptotic} discovered that as $N\to \infty$ the discrete
measures \eqref{eq:1} in the weight space $\mathbb{R}^{n}$ of $su(n+1)$ converge
weakly to some continuous measure on the main Weyl chamber.

In this paper we consider $\gf=B_n$ and tensor powers of the spinor representation $L^{\omega_n}$. We prove the weak convergence of measures
\eqref{eq:1} on the main Weyl chamber of the weight space for $B_{n}$ to the measure induced by the $G$-invariant
Euclidean measure on $\gf$.

In Theorem \ref{local} we consider the small domain $U(N)\subset \mathbb{R}^n$ the volume of which goes to zero as $N\to\infty$. We show that the probability density function of the limit measure is given by the formula :
\begin{equation*}\phi\left(\{x_i\}\right)=
\prod_{i< j}(x_i^2-x_j^2)^2\prod_{l=1}^{n}x_l^2
\exp\left(-\frac{1}{2}\sum_kx_k^2\right)\cdot
\frac{2^{2n}n!}{(2n)!(2n-2)!\dots 2!},
\end{equation*}
where $x_{i}=\frac{1}{\sqrt{N}}a_i$ and $a_{i}=\lambda_i+\rho_i$ are the shifted Euclidean coordinates on weight
space of $B_{n}$, where $\rho$ is the Weyl vector.

In Theorem \ref{global} we show that that the same holds for every $n$-orthotope in the main Weyl chamber.

In Theorem \ref{weak} we show that the measures converge weakly on the entire main Weyl chamber.

Kerov's proof for the $A_{n}$-case was based upon the hook formulas and Young diagrams for tensor
product decomposition multiplicities and dimensions of representations. In this it was similar to
the famous Vershik-Kerov \cite{vershik1977asymptotics,vershik1985asymptotic} and Logan-Shepp
\cite{logan1977variational} result on the limit shape of Young diagrams. Relation of the Young
diagrams limit shape to the random matrices was established in the papers
\cite{okounkov2000random,borodin2000asymptotics}. The weak convergence was established from the
comparison of the discrete probability measure to multinomial distribution and the use of de
Moivre-Laplace theorem.

Unfortunately, there are no analogues of hook formulas for the other Lie algebras. Thus we use
combinatorial formula for tensor product decomposition coefficients \cite{kulish2012tensor} and Weyl dimension formula to
prove the analogue of de Moivre-Laplace theorem \cite{gnedenko2017theory,stirling2014notes} for an
arbitrary $n$-orthotope in the main Weyl chamber, and then we apply a criterion for weak convergence
of probability measures \cite{bogachev2016slabaya}.

The paper is organized as follows. In Section \ref{sec:tens-prod-decomp} we fix the notations and give the definition
of probability measures on subsets of dominant integral weights. In Section \ref{sec:tens-prod-decomp-1} we present the formula for
tensor product decomposition multiplicities for tensor powers of spinor representation of the
algebra $B_{n}$. In Section \ref{sec:limit-infin-tens} we take limit of infinite tensor
power and prove the convergence of discrete probability measures on subsets of dominant integral
weights to a continuous probability measure on the main Weyl chamber. In the Conclusion we discuss the
connection to  the invariant measure and further work.

\section{Tensor product decomposition and definition of probability measure}
\label{sec:tens-prod-decomp}

Consider a simple Lie algebra $\gf=B_{n}$ of rank $n$. Denote simple roots of $\gf$ by $\alpha_{1},\dots,
\alpha_{n}$ and fundamental weights by $\omega_{1},\dots \omega_{n}$,
$(\alpha_{i},\omega_{j})=\delta_{ij}$. We denote an irreducible highest weight representation of
$\gf$ with the highest weight $\lambda$ by $L^{\lambda}$.

The root system is denoted by $\Delta=\{\pm e_i \pm e_j|_{i\neq j}\}\cup\{\pm e_i\}$, 
$\Delta^{+}=\{e_i+e_j|_{i\leq j}\}\cup\{e_i\}\cup\{e_j-e_i|_{j\leq i}\}$ is the subset of positive
roots, and $\rho=\omega_1+\dots+\omega_n$ is the Weyl vector.

We consider the multiplicity function for the tensor power of the
last fundamental module $\omega_n=\frac{1}{2}(e_1+\dots e_n)$.
We decompose $N$-th tensor power of $L^{\omega_{n}}$ into the direct sum of the irreducible
representations:
\begin{equation}
\label{eq:6}
  \left(L^{\omega_{n}}\right)^{\otimes N} = \bigoplus_{\lambda\in P^{+}(\omega_{n},N)} M^{\omega_{n},N}_{\lambda} L^{\lambda},
\end{equation}
where $P^{+}(\omega_{n},N)$ is the subset of dominant integral weights of the (reducible) representation
$\left(L^{\omega_{n}}\right)^{\otimes N}$. 

The relation \eqref{eq:1} gives us a probability measure on
$P^{+}(\omega_{n},N)$ with the probability density function
\begin{equation}
  \label{eq:2}
  \mu_{N}(\lambda) = \frac{M^{\omega_{n},N}_{\lambda} \dim L^{\lambda}}{\left(\dim L^{\omega_{n}}\right)^{N}}.
\end{equation}
It is easy to see that $\sum_{\lambda\in P(\omega_n,N)} \mu_N(\lambda) =1$, since the dimension of the
l.h.s. of the equation \eqref{eq:6} is equal to $\left(\dim L^{\omega_n}\right)^{N}$.

In present paper we will study this measure in the limit $N\to\infty$ for $n$ fixed.

\section{Multiplicity formula for tensor product decomposition}
\label{sec:tens-prod-decomp-1}

The papers
\cite{kulish2012tensor,kulish2012tensorb2,kulish2012multiplicity,kulish2012multiplicityJPC} proposed a method
for decomposing tensor powers of fundamental module of smallest dimension
for the simple Lie algebras of series
$A_{n}$ and $B_{n}$.

Instead of focusing on the multiplicity function $M^{\omega,N}_{\lambda}$, that is defined on the
set of dominant weights $P^{+}$, this method is aimed at finding the expression for the extended
multiplicity function $\tilde{M}^{\omega,N}_{\lambda}$ defined on the whole weight lattice $P$
as:
\begin{equation}
\tilde{M}^{\omega,N}_{w(\lambda +\rho) -\rho}|_{ w \in W} =\epsilon \left( w\right)M^{\omega,N}_{\lambda}.
\label{mult function}
\end{equation}

If the expression for $\tilde{M}^{\omega,N}_{\lambda}$ is obtained, we could find
$M^{\omega,N}_{\lambda}$ as its contraction to the set of dominant weights $P^{+}$ that lies in the main Weyl chamber.

It was shown in \cite{kulish2012multiplicity,kulish2012multiplicityJPC} that the extended multiplicity function $\tilde{M}^{\omega,N}_{\lambda}$
is a solution of the set of recurrence relations:
\begin{equation}
\sum_{\xi \in P}\tilde{M}^{\omega,N}_{\xi}e^{ \xi }
  =\mathcal{N}\left( L_{\mathfrak{g}}^{\left( \omega\right) }\right) \sum_{\gamma \in P}\tilde{M}^{\omega,N-1}_{\gamma}e^{ \gamma },  \label{our t-fusion}
\end{equation}
where $\mathcal{N}\left( L_{\mathfrak{g}}^{\left( \omega\right) }\right)$ is a weight diagram of the module $L_{\mathfrak{g}}^{\left( \omega\right)}$.

It has also been proven that in case of fundamental modules of smallest dimension the solution of \eqref{our t-fusion} is uniquely determined by the
requirements of anti invariance with respect to the Weyl group transformations and the boundary
conditions.

In case of $B_n$ algebra this method allowed us to obtain the multiplicity function for the tensor power of the last fundamental module $\omega_n=\frac{1}{2}(e_1+\dots + e_n)$. This expression has an explicit dependence on $N$:
 \begin{equation}
    \tilde{M}^{\omega_{n},N}_{\lambda(a_1\dots a_{n})}=
   \prod_{k=0}^{n-1}\frac{\left(
N+2k\right) !}{2^{2k}\left( \frac{N+a_{k+1}+2n-1}{2}\right) !\left( \frac{%
N-a_{k+1}+2n-1}{2}\right) !}\prod_{l=1}^{n}a_{l}\prod_{ i<j %
}\left( a_{i}^{2}-a_{j}^{2}\right)\rule{0mm}{8mm}
   \label{bn}
   \end{equation} 
here $\{a_i\}$ are the coordinates of $\lambda$ in the basis
$\left\{\{\vec{\tilde{e_{i}}}\}: \vec{\tilde{e_{i}}}\parallel \vec{e_i},
  |\tilde{e_{i}}|=|\frac{e_i}{2}|\right\}$ with the center of coordinates shifted to
$-\rho=-\omega_1\dots -\omega_n$. The factors in the numerator vanish at the boundaries of the
shifted Weyl chambers and the denominator provides that $\tilde{M}^{\omega_{1},N}_{\lambda}$
satisfies the boundary conditions and also ensures that the whole expression is anti invariant
w.r.t. Weyl group transformations.
   
Note that there are two congruence classes of weights, one is parametrized by even values of
$a_{i}$ while another by odd. The class is determined by the parity of $N$. For $N$ even we get
$a_{i}$ odd and vice versa.

The expression after $(N+2k)!$ in the numerator is connected to the Weyl dimension formula for the irreducible module
$L^{\lambda}$:
\begin{equation}
  \label{eq:5}
  \dim L^{\lambda} = \prod_{\alpha\in\Delta^{+}} \frac{(\lambda+\rho,\alpha)}{(\rho,\alpha)},
\end{equation}
which for the case of $B_n$ module has the form:
\begin{equation}
\label{weyl}
\dim L^{\lambda}=\frac{\prod_{i< j}(a_i^2-a_j^2)\prod_{l=1}^{n}a_l}{(2n)!(2n-2)!\dots 2!}\cdot 2^{-n^2+2n}n!.
\end{equation}
Thus we obtain the discrete probability measure  with density function (or the probability mass function):
\begin{multline}
\label{meas}
\mu_N(\lambda)=
\mu_N(\{a_i\})
=\frac{\tilde{M}^{\omega_{n},N}_{\lambda(a_1\dots a_{n})}\dim L^{\lambda}}{(2^n)^N}=\\
=
\prod_{k=0}^{n-1}\frac{\left(
N+2k\right) !}{2^{2k}\left( \frac{N+a_{k+1}+2n-1}{2}\right) !\left( \frac{%
N-a_{k+1}+2n-1}{2}\right) !}\prod_{i< j}(a_i^2-a_j^2)^2\prod_{l=1}^{n}a_l^2\cdot\frac{2^{-n^2+2n-nN}n!}{{(2n)!(2n-2)!\dots 2!}}.
\end{multline}

\section{Limit of the infinite tensor product for a finite-rank algebra}
\label{sec:limit-infin-tens}
Let now the $n$-dimensional random vector $X$ be distributed according to the discrete measure with
density function (or, more correctly, the probability mass function) \eqref{meas}:
$X\sim \mu_N(\{a_i\}): U\subset P^{+}\to \mathbb{R}$. Here $P^{+}$ is the dominant weight lattice.
$$\begin{cases}

0\leq\mu_N(U)\leq 1  \\ \mu_N(P^{+})=1 

\end{cases}$$

We fix $n$ in the formula \eqref{meas} and study the limit
$N\to\infty$ to see that this probability mass function converges to continuous probability density
function $\phi\{x_i\}:U\subset\mathbb{R}^{n}\to \mathbb{R}$.

Firstly, we embed $P^{+}\subset\mathbb{R}^{n}$ the following way. Let us associate to the weight $\lambda$ in $P^{+}$ with coordinates $\{a_i\}$ the domain $U_{a}=\cup_i[a_i-1,a_i+1)\in \mathbb{R}^{n}$. There will only be one weight inside this domain. Consider the probability mass function of vector $X$:
\begin{multline}
\mu_N(\lambda)=
p_X(\lambda)=\mathbf{P}\{X=\lambda\}=\mathbf{P}\{X\in U_{a}\}=
\mathbf{P}\{a_i-1\leq X_i< a_i+1\}=\\=\mathbf{P}\left\{\frac{1}{\sqrt{N}}(a_i-1)\leq \frac{1}{\sqrt{N}}X_i< \frac{1}{\sqrt{N}}(a_i+1)\right\}=
\mathbf{P}\left\{\frac{1}{\sqrt{N}}X\in U_{a}(N)\right\}.
\end{multline}
The volume of the rescaled domain $U_{a}(N)$ goes to zero as $N\to\infty$.

Then as $N\to\infty$ we would expect convergence on this domain
\begin{equation}
\left|p_X(\lambda)\cdot\left( \frac{\sqrt{N}}{2}\right)^n-\phi\left(\left\{\frac{1}{\sqrt{N}}a_i\right\}\right)\right|\longrightarrow 0
\end{equation}

\begin{theorem}
\label{local}
Let $X\sim \mu_N(\{a_i\})$ and $C_N$ be nondecreasing such that $\displaystyle \lim_{N\longrightarrow\infty} C_N/N^{\frac{1}{6}}=0$. Then
\begin{equation}
\max_{|a_i+2n-1|<\sqrt{N}\cdot C_N}\left|\frac{p_X(\lambda)}{\phi\left(\{x_i\}\right)}\left(\frac{\sqrt{N}}{2}\right)^n-1\right|=\mathcal{O}\left(\frac{C_N^3}{\sqrt{N}}\right),
\end{equation}
where $x_i=\frac{1}{\sqrt{N}}a_i 
$ and
\begin{equation}
  \phi\left(\{x_i\}\right)= \prod_{i< j}(x_i^2-x_j^2)^2\prod_{l=1}^{n}x_l^2 \exp\left(-\frac{1}{2}\sum_kx_k^2\right)\cdot   \frac{2^{2n}n!}{(2n)!(2n-2)!\dots 2!}
\end{equation}
\end{theorem}

\begin{proof}
Consider the factor in \eqref{meas} that depends on $N$:
\begin{equation}
\label{in}
I_{N,n}=\prod_{k=0}^{n-1}\frac{\left(
N+2k\right) !}{2^{2k}\left( \frac{N+a_{k+1}+2n-1}{2}\right) !\left( \frac{%
N-a_{k+1}+2n-1}{2}\right) !}
\end{equation}

We will treat the numerator of $I_{N,n}$ and the denominator
\begin{equation}
  \label{eq:8}
    D_{N,n}=\prod_{k=0}^{n-1}{2^{2k}\left( \frac{N+a_{k+1}+2n-1}{2}\right) !\left( \frac{%
N-a_{k+1}+2n-1}{2}\right) !},
\end{equation}
separately. 

To obtain asymptotics of 
\eqref{in} we will be using the Stirling formula for
factorials
\begin{equation}
  \label{stirl}
  N!\approx \sqrt{2\pi}\; \mathrm{exp}\left(N\ln N -N+\frac{1}{2}\ln N\right)\left(1+\mathcal{O}\left(\frac{1}{N}\right)\right),
\end{equation}
and assuming that $n\ll N$ we will use the following expansion of logarithm to the order of $\frac{1}{N^{2}}$:
\begin{equation}
\ln(N+i)=\ln N\left(1+\frac{i}{N}\right)=\ln N +\frac{i}{N}-\frac{i^2}{2N^2}+\mathcal{O}{\left(\frac{1}{N^3}\right)}
\end{equation}
\begin{equation}
S=\sum_{i=1}^{n}(N+i)\ln(N+i)=\sum_{i=1}^{n}\frac{1}{N}\left(-\frac{i^{3}}{2N}+\frac{i^2}{2}+N(1+\ln N)i\right)+Nn\ln N+\mathcal{O}\left(\frac{1}{N^{2}}\right).
\end{equation}
Since the first term of the sum is of order $\frac{1}{N^{2}}$ we can include it inside the error:

\begin{equation}
\label{lnexp}
S=\sum_{i=1}^{n}\frac{1}{N}\left(
\frac{i^2}{2}+N(1+\ln N)i\right)+Nn\ln N+\mathcal{O}\left(\frac{1}{N^{2}}\right).
\end{equation}
Then as $N\longrightarrow\infty$ we apply the Stirling formula to the numerator of $I_{N,n}$:
\begin{multline}
\prod_{k=0}^{n-1}\left(N+2k\right) !=
\prod_{k=1}^{n}\left(
N+2k-2\right) !\simeq\\
\simeq
\prod_{k=1}^{n}\sqrt{2\pi}
\exp
\left[
(N+2k-2)\ln(N+2k-2)-(N+2k-2)
\right]\prod_{k=1}^{n}(N+2k-2)^{\frac{1}{2}}
\left(1+\mathcal{O}\left(\frac{1}{N}\right)\right)=\\
=(\sqrt{2\pi})^n
\exp
\left[\sum_{k=1}^{n}
(N+2k-2)\ln(N+2k-2)-(N+2k-2)
\right]\prod_{k=1}^{n}(N+2k-2)^{\frac{1}{2}}\left(1+\mathcal{O}\left(\frac{1}{N}\right)\right).
\end{multline}
Using \eqref{lnexp} we will now expand the sum under the exponent:
\begin{multline}
\label{num}
\prod_{k=0}^{n-1}\left(N+2k\right)! \simeq\\\simeq(\sqrt{2\pi})^n
\exp
\left[\sum_{k=1}^{n}
\left(
\frac{(2k-2)^2}{2N}+(1+\ln N)(2k-2)
+N\ln N
-(N+2k-2)\right)
+\mathcal{O}\left(\frac{1}{N^2}\right)\right]\cdot\\
\cdot
\prod_{k=1}^{n}(N+2k-2)^{\frac{1}{2}}
\left(1+\mathcal{O}\left(\frac{1}{N}\right)\right)=\\
=(\sqrt{2\pi})^n
\exp
\left[
\frac{1}{N}\frac{4n^3-6n^2+2n}{6}
+n(n-1)-n(N+n-1)\right]N^{n(n-1)}N^{nN}
\cdot\\
\cdot
\prod_{k=1}^{n}(N+2k-2)^{\frac{1}{2}}
\left(1+\mathcal{O}\left(\frac{1}{N^2}\right)\right)
\left(1+\mathcal{O}\left(\frac{1}{N}\right)\right)=\\
=
(\sqrt{2\pi})^n
\exp
\left[-nN+
\frac{1}{N}\frac{4n^3-6n^2+2n}{6}
\right]N^{n(n-1)}N^{nN}\prod_{k=1}^{n}(N+2k-2)^{\frac{1}{2}}
\left(1+\mathcal{O}\left(\frac{1}{N}\right)\right).
\end{multline}
Now let us expand the denominator $D_{N,n}$ of $I_{N,n}$. We can expand logarithms such as $\ln(1+\frac{a_k+2n-1}{N})$  in the interval $\frac{|a_k+2n-1|}{N}<\frac{C_k}{\sqrt{N}}\ll 1$. We will choose $C_N$ that satisfies this condition for the largest $a_k$. Therefore the value of $\frac{|a_k+2n-1|}{N}$ will be bounded for all $a_k$ as well as the value of $\frac{|-a_k+2n-1|}{N}$. 

Applying the Stirling formula for the denominator of $I_{N,n}$ we get:
\begin{multline}
D_{N,n}=\prod_{k=1}^{n}2^{2(k-1)}\left( \frac{N+a_{k}+2n-1}{2}\right) !\left( \frac{%
N-a_{k}+2n-1}{2}\right) !\simeq2^{n(n-1)}(\sqrt{2\pi})^{2n}\cdot\\
\cdot\prod_{k=1}^{n}\exp\left[\frac{N+a_{k}+2n-1}{2}\ln\left(\frac{N+a_{k}+2n-1}{2}\right)-\frac{N+a_{k}+2n-1}{2}\right]\prod_{k=1}^{n}
\left(
\frac{N+a_{k}+2n-1}{2}
\right)^{\frac{1}{2}}\cdot\\
\cdot
\prod_{k=1}^{n}\exp\left[\frac{N-a_{k}+2n-1}{2}\ln\left(\frac{N-a_{k}+2n-1}{2}\right)-\frac{N-a_{k}+2n-1}{2}\right]\prod_{k=1}^{n}
\left(
\frac{N-a_{k}+2n-1}{2}
\right)^{\frac{1}{2}}
\end{multline}
This expansion has the error bound of $\left(1+\mathcal{O}\left(\frac{1}{N}\right)\right)$. To simplify calculations we will combine some factors into one factor $Z(n,N)$. So we get for the denominator:
\begin{multline}
D_{N,n}=\underbrace{2^{n(n-1)}(\sqrt{2\pi})^{2n}
\exp\left(-n(N+2n-1)\right)\prod_{k=1}^{n}
\left(
\frac{N-a_{k}+2n-1}{2}
\right)^{\frac{1}{2}}
\left(
\frac{N+a_{k}+2n-1}{2}
\right)^{\frac{1}{2}}}_{Z(n,N)}
\cdot\\
\cdot
\exp\left[\sum_{k=1}^{n}\frac{N-a_{k}+2n-1}{2}\ln\left(\frac{N-a_{k}+2n-1}{2}\right)+\frac{N+a_{k}+2n-1}{2}\ln\left(\frac{N+a_{k}+2n-1}{2}\right)\right]
\cdot
\\
\cdot
\left(1+\mathcal{O}\left(\frac{1}{N}\right)\right)=\\
=Z(n,N)\cdot
\exp
\left[\frac{1}{2}\sum_{k=1}^{n}(N+a_k+2n-1)\ln(N+a_k+2n-1)+(N-a_k+2n-1)\ln(N-a_k+2n-1)\right]\cdot\\
\cdot \exp\left[-\sum_{k=1}^{n}(N+2n-1)\ln 2\right]\left(1+\mathcal{O}\left(\frac{1}{N}\right)\right)
\end{multline}
Using \eqref{lnexp} we expand the sum under the exponent and get the approximate expression for the denominator:
\begin{multline}
  D_{N,n}\simeq 2^{-n(N+2n-1)}Z(n,N)\cdot\\
  \cdot
\exp
\left[\frac{1}{2N}\sum_{k=1}^{n}\left(\frac{(a_k+2n-1)^2}{2}+N(1+\ln N)(a_k+2n-1)+N^2\ln N\right)+\mathcal{O}\left(\frac{C_N^3}{\sqrt{N}}\right)\right]\cdot\\
\cdot 
\exp
\left[\frac{1}{2N}\sum_{k=1}^{n}\left(\frac{(-a_k+2n-1)^2}{2}+N(1+\ln N)(-a_k+2n-1)+N^2\ln N\right)+\mathcal{O}\left(\frac{C_N^3}{\sqrt{N}}\right)\right]\cdot\\
\cdot \left(1+\mathcal{O}\left(\frac{1}{N}\right)\right)=\\=
2^{-n(N+2n-1)}Z(n,N)
\exp
\left[\sum_{k=1}^{n}\frac{a_k^2}{2N}\right]\exp
\left[\frac{4n^3-4n^2+n}{2N}+n(2n-1)+n(2n-1)\ln N+ nN\ln N \right]\cdot\\
\cdot\left(1+\mathcal{O}\left(\frac{C_N^3}{\sqrt{N}}\right)\right)\left(1+\mathcal{O}\left(\frac{1}{N}\right)\right)
\end{multline}
Extracting the factors from $Z(n,N)$ we get:
\begin{multline}\label{den}
D_{N,n}\simeq 2^{-nN-n^2}(\sqrt{2\pi})^{2n}N^{nN+n(2n-1)}\exp \left[\sum_{k=1}^{n}\frac{a_k^2}{2N}\right]\exp
\left[-nN+\frac{4n^3-4n^2+n}{2N}\right]
\cdot\\
\cdot
\prod_{k=1}^{n}
\left(
\frac{N-a_{k}+2n-1}{2}
\right)^{\frac{1}{2}}
\left(
\frac{N+a_{k}+2n-1}{2}
\right)^{\frac{1}{2}}\left(1+\mathcal{O}\left(\frac{C_N^3}{\sqrt{N}}\right)\right)\left(1+\mathcal{O}\left(\frac{1}{N}\right)\right).
\end{multline}
Then, dividing \eqref{num} by \eqref{den} we obtain:
\begin{multline}
I_{N,n}\simeq\frac{2^{nN+n^2}}{(\sqrt{2\pi})^n}N^{-n^2}\exp \left[-\sum_{k=1}^{n}\frac{a_k^2}{N}\right]\exp\left[\frac{-8n^3-6n^2-n}{3N}\right]
\underbrace{\prod_{k=1}^{n}\frac{(N+2k-2)^{\frac{1}{2}}}{\left(
\frac{N-a_{k}+2n-1}{2}
\right)^{\frac{1}{2}}
\left(
\frac{N+a_{k}+2n-1}{2}
\right)^{\frac{1}{2}}}}_{d_{N,n}}\cdot\\
\cdot
\left(1+\mathcal{O}\left(\frac{1}{N}\right)\right)
\left(1+\mathcal{O}\left(\frac{C_N^3}{\sqrt{N}}\right)\right)\left(1+\mathcal{O}\left(\frac{1}{N}\right)\right).
\end{multline}
We denote the product over $k$ by $d_{N,n}$ and expand $d_{N,n}$ separately:
\begin{multline}
d_{N,n}=\exp\left[\frac{1}{2}\sum_{k=1}^{n}\left\{
\ln\left(N\left(1+\frac{2k-2}{N}\right)\right)-
\ln\left(N\left(\frac{1}{2}+\frac{a_k+2n-1}{2N}\right)\right)-\right.\right.\\
\left.\left.
-\ln\left(N\left(\frac{1}{2}+\frac{-a_k+2n-1}{2N}\right)\right)
\right\}\right]\simeq\\
\simeq
\exp\left[
\frac{1}{2N}\sum_{k=1}^{n}\left(2(k-1)-(a_k+2n-1)-(-a_k+2n-1)-N\ln N-2N\ln{\frac{1}{2}}\right)+\mathcal{O}\left(\frac{C_N^2}{N^2}\right)
\right]\cdot\\
\cdot
\left(1+\mathcal{O}\left(\frac{1}{N}\right)\right)\simeq\\
\simeq
2^n
N^{-\frac{n}{2}}\exp\left[\frac{-n^2-n}{N}\right]\left(1+\mathcal{O}\left(\frac{C_N^2}{N^2}\right)\right)\left(1+\mathcal{O}\left(\frac{1}{N}\right)\right)
\end{multline}
and choosing the largest of the error bounds we obtain:
\begin{multline}
I_{N,n}\simeq\frac{2^{nN+n^2+n}}{(\sqrt{2\pi})^n}N^{-n^2-\frac{n}{2}}
\exp \left[-\sum_{k=1}^{n}\frac{a_k^2}{2N}\right]\exp\left[\frac{-8n^3-6n^2-n}{3N}\right]
\left(1+\mathcal{O}\left(\frac{C_N^3}{\sqrt{N}}\right)\right).
\end{multline}
Finally, for the measure density \eqref{meas} we get:
\begin{multline}
\label{asmeas}
\lim_{N\longrightarrow \infty}p_X(\lambda)
=
\left(\frac{1}{N}\right)^{n^2+n/2}\frac{2^{3n}}{(\sqrt{2\pi})^n}\frac{n!}{{(2n)!(2n-2)!\dots 2!}}
\prod_{i<j}(a_i^2-a_j^2)^2\prod_{l=1}^{n}a_l^2\exp \left[-\sum_{k=1}^{n}\frac{a_k^2}{N}\right]=\\
=
\left(\frac{2}{\sqrt{N}}\right)^n
\phi(\{x_i\})\left(1+\mathcal{O}\left(\frac{C_N^3}{\sqrt{N}}\right)\right).
\end{multline} 
where $x_i=\frac{1}{\sqrt{N}}a_i 
$ and
\begin{equation}
\label{phi}
\phi(\{x_i\})
=\frac{2^{2n}n!}{(\sqrt{2\pi})^n(2n)!(2n-2)!\dots 2!}
\prod_{i< j}(x_i^2-x_j^2)^2\prod_{l=1}^{n}x_l^2
\exp\left(-\frac{1}{2}\sum_kx_k^2\right).
\end{equation}
\end{proof}
Above we have proven the local theorem for the small domain $U_{a}(N)$ the volume of which goes to
zero as $N\longrightarrow\infty$. We will now prove the global theorem.
\begin{theorem}
\label{global}
Let $X\sim \mu_N(\{a_i\})$, then for every n-orthotope 
$H_n=\{c_1,d_1\}\times\{c_2,d_2\}\times
\dots\times\{c_n,d_n\}\subset P^{+}$ where all $\{c_i\}<\{d_i\}$ are fixed real numbers, then 
\begin{equation}
\lim_{N\longrightarrow\infty} \mathbf{P}\left\{c_i\leq \frac{1}{\sqrt{N}}X_i< d_i\right\}=\int_{H_n}\phi(\{x_i\})dx_1\dots dx_n
\end{equation}
\begin{proof}
We first use the triangle inequality to obtain
\begin{multline}
\left|\mathbf{P}\{c_i\leq  \frac{1}{\sqrt{N}}X_i< d_i\}-\int_{H_n}\phi(\{x_i\})dx_1\dots dx_n\right|\leq
\left|
\sum_{a_i=\lceil c_i \sqrt{ N}\rceil}
^{\lceil d_i \sqrt{ N}-1\rceil}
p_X(\lambda)-
\int_{H_n}\phi(\{x_i\})dx_1\dots dx_n
\right|\leq\\
\leq
\left|
\sum_{a_i=\lceil c_i \sqrt{N}\rceil}
^{\lceil d_i \sqrt{ N}-1\rceil}
\left(p_X(\lambda)-
\left( \frac{2}{\sqrt{N}}\right)^n\phi\left( \frac{\{a_i\}}{\sqrt{N}}\right)\right)
\right|+\\
+
\left|
\sum_{a_i=\lceil c_i \sqrt{ N}\rceil}
^{\lceil d_i \sqrt{ N}-1\rceil}
\left( \frac{2}{\sqrt{N}}\right)^n\phi\left(\frac{\{a_i\}}{\sqrt{N}}\right)-
\int_{H_n}\phi(\{x_i\})dx_1\dots dx_n
\right|
\end{multline}
The second term is the difference between the integral and its Riemann sum, therefore it goes to
zero. For the first term, let $c>\max(a,b)$ then
\begin{multline}
\max_{|a_i+2n-1|<\sqrt{N}\cdot c}\left|p_X(\lambda)-
\left( \frac{2}{\sqrt{N}}\right)^n\phi\left( \frac{\{a_i\}}{\sqrt{N}}\right)\right|
=\\
=
\max_{|a_i+2n-1|<\sqrt{N}\cdot c}\left|\frac{p_X(\lambda)\left(\frac{\sqrt{N}}{2}\right)^n}{\phi\left( \frac{\{a_i\}}{\sqrt{N}}\right)}-1\right|
\cdot
\left( \frac{2}{\sqrt{N}}\right)^n
\phi\left( \frac{\{a_i\}}{\sqrt{N}}\right)\leq\\
\leq
\mathcal{O}\left(\frac{1}{N}\right)\cdot
\mathcal{O}\left(\frac{1}{N^{n/2}}\right)=
\mathcal{O}\left(\frac{1}{N^{n/2+1}}\right)
\end{multline}
There are $n$ sums of these expressions and each has $\mathcal{O}(\sqrt{N})$ summands. Therefore
there is total $\mathcal{O}\left(N^{n/2}\right)$ summands, and the sum is in the order of
$\mathcal{O}\left(\frac{1}{N}\right)$.
\end{proof}
\end{theorem}
We can in fact check that the obtained density function \eqref{phi} defines a probability measure by integrating it over the main Weyl chamber. We will use  the following integral, studied by Macdonald in \cite{macdonald1982some}:
\begin{equation}
\label{mac}
\frac{1}{(2\pi)^{n/2}}\int\dots\int\prod_{i=1}^{n}(2|x_i|^2)^{\gamma}\prod_{1\leq i<j\leq n}\left|x_i^2-x_j^2\right|^{2\gamma} \exp\left(-\sum_{k=1}^{n}\frac{|x_k|^2}{2}\right)dx_1\dots dx_n=\prod_{j=1}^{n}\frac{\Gamma(1+2j\gamma)}{\Gamma(1+\gamma)}
\end{equation}
It is easy to see that if we integrate \eqref{phi} over the whole space, following Macdonald, we
will obtain the factor $2^nn!$. But $2^nn!$ is the order of the Weyl group of $B_n$, so integration
over the main Weyl chamber will give us exactly 1.
\begin{theorem}
\label{weak}
The sequence of discrete probability measures with densities $\mu_N(\lambda)$ on the main Weyl chamber converges weakly to the continuous measure $\mu$ with density $\phi(\{x_i\})$.
\begin{proof}
  We will use the following criterion of weak convergence of measures \cite{bogachev2016slabaya}:\\
  \textit{Let }$\mathcal{E}$\textit{ be the class of open sets in metric space} $X$\textit{, which is closed under finite intersection, and every open set can be represented as countable or finite union of sets from} $\mathcal{E}$\textit{. Let $\mu_N, \mu$ be probability Borel measures such that} $\mu_N(E)\longrightarrow \mu(E)$\textit{ for all }$E\in \mathcal{E}$\textit{.Then the sequence $\mu_N$ converges weakly to} $\mu$.\\[2mm]
  Since $\mathcal{E}$ can be comprised of certain $n$-orthotopes, and for every $n$-orthotope
  $\mu_N(H_n)\longrightarrow \mu(H_n)$ following Theorem \ref{global}, the weak convergence
  $\mu_N\Rightarrow \mu$ is proven.
\end{proof}
\end{theorem}

\begin{example}
We illustrate the result with a simple example. In Fig. \ref{fig:b2-example} we plot values of
probability mass function $p_X(\lambda)$ (indicated by dots) and the probability density function
$\phi(\{x_i\})$ for the algebra $B_{2}$($so(5)$) and  $N=15$ power of second fundamental representation $L^{\omega_{2}}$. 
\end{example}
\begin{figure}[hbt]
    \includegraphics[width=10cm]{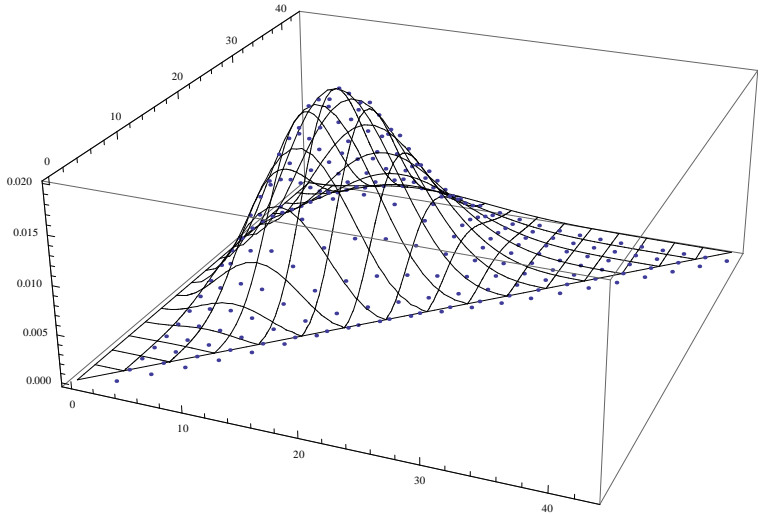}
    \caption{Values of probability mass function $p_X(\lambda)$ (indicated by dots) and the probability density function $\left(\frac{2}{\sqrt{N}}\right)^n\cdot\phi(\{x_i\})$ for $n=2$ and $N=15$ in the main Weyl chamber in the rescaled axes $x_i=\frac{a_i}{\sqrt{N}}$}.
  \label{fig:b2-example}
  \end{figure}

\section*{Conclusion and outlook}
\label{sec:conclusion-outlook}
The Lie algebra $so(2n+1)$  is, naturally, the Lie algebra of $(2n+1)\times(2n+1)$ orthogonal matrices, i.e $A^t=A^{-1}$. The Killing form on $so(2n+1)$ is proportional to bilinear form $\mathrm{tr}(A^tB)$ which is symmetric, positive definite and defines a Riemannian metric on $so(2n+1)$.

The corresponding Riemannian integration measure is $SO(2n+1)$  - invariant (with respect to $A\mapsto gAg^{-1}$). It is well-known that the integral of $SO(2n+1)$-invariant function over $so(2n+1)$ can be written in terms of its radial part (integration over eigenvalues):
\begin{equation}
\int f(A)dA=\int_{\mathbb{R}^n}f\left(\mathrm{diag}(a_1,\dots a_n,-a_n,\dots,-a_1)\right)\cdot
 \frac{2^{n^{2}}}{\pi^{n}n!}
 \prod_{1\leq i<j\leq n}
  \left(a_{i}^{2}-a_{j}^{2}\right)^{2}\prod_{k=1}^{n}a_{k}^{2}\;da_1\dots da_k
\end{equation}
Remarkably, the radial part of this measure gives exactly non-Gaussian factors in the limit of the Plancherel measure, exactly as in Kerov's work on $su(n+1)$\cite{kerov1986asymptotic}.
We expect this will hold for other simple Lie algebras.

\underline{Remark.} There is a simple relation between invariant measure on $so(2n+1)$ and the Haar measure on $SO(2n+1)$. If $dg$ is the Haar measure normalized as $\int_G dg=1$ and $f$ is a function on $G$ invariant with respect to conjugations, we have
\begin{equation}
  \label{eq:3}
  \int f(g) \mathrm{dHaar} = \frac{2^{n^{2}}}{\pi^{n}n!}\int_{[0,\pi]^{n}}f(\Theta_{1},\dots,\Theta_{n})\prod_{1\leq i<j\leq n}
  \left(\cos(\Theta_{i})-\cos(\Theta_{j})\right)^{2}\prod_{k=1}^{n}\sin^{2}\frac{\Theta_{k}}{2} d\Theta_{1}\dots d\Theta_{n}
\end{equation}
where $\lambda_{j}=e^{i\Theta_{j}}$ are distinct eigenvalues of generic $g\in G$.

Let $f_{\epsilon}(g)$ be a family of such functions supported on neighborhoods $U_{\epsilon}$ of $\Theta=0$ such that $\frac{1}{\epsilon}U_{\epsilon}\to W\subset \mathbb{R}_+^n$ as $\epsilon\to 0$, then
\begin{equation}
  \label{eq:4}
 \int_{[0,\pi]^{n}}f(\Theta)d\mu(\Theta)=\int_W \epsilon^{2n^2+n}f(\epsilon a)(1+\mathcal{O}(\epsilon))d\mu_0(a)
\end{equation}
where $d\mu(\Theta)$ is the radial part of the Haar measure as in \eqref{eq:3} and $d\mu_0$ is the
radial part of the measure on the Lie algebra $so(2n+1)$.

\section*{Acknowledgements}
\label{sec:acknowledgements}

We thank prof. Nikolai Reshetikhin for his guidance and prof. Fedor Petrov for valuable discussions.
This work is supported by the Russian Science Foundation under grant No. 18-11-00297.
% Olga Postnova is supported by the Russian Science Foundation under grant No. 18-11-00297.
% Anton Nazarov is supported by the Russian Foundation for Basic Research under grant No. 18-01-00916 
\bibliography{measures,bibliography,listing}{} 
\bibliographystyle{utphys}

\end{document}